\documentclass[reqno,a4paper]{amsart}
\usepackage[utf8]{inputenc}

\usepackage{amsmath,amssymb,mathtools,hyperref}

\title{Quasimodularity of the $k$th Residual Cranks}
\author[T.~Morrill]{Thomas Morrill}
\address{School of Science, The University of New South Wales (Canberra), ACT, Australia}
\email{t.morrill@adfa.edu.au}
\author[A.~Simoni\v{c}]{Aleksander Simoni\v{c}}
\address{School of Science, The University of New South Wales (Canberra), ACT, Australia}
\email{a.simonic@student.adfa.edu.au}
\date{\today}

\newcommand{\nklist}[2]{{#1}_{1}, \ldots, {#1}_{#2}}
\newcommand{\aqprod}[3]{\left({#1};{#2}\right)_{#3}}
\newcommand{\ocr}{\overline{cr}}
\newcommand{\de}[0]{\mathrel{\mathop:}=}
\newcommand{\dif}[1]{\mathrm{d}#1}
\newcommand{\ie}[0]{\mathrm{i}}
\newcommand{\lcm}[1]{\mathrm{lcm}\left(#1\right)}
\newcommand{\OP}{\overline{P}}

\newtheorem{theorem}{Theorem}
\newtheorem{lemma}{Lemma}
\newtheorem{cor}{Corollary}

\newcommand{\CC}{\mathbb{C}}

\usepackage{graphicx}

\begin{document}

\begin{abstract}
We establish quasimodularity for a family of residual crank generating functions defined on overpartitions. We also show that the second moments of these $k$th residual cranks admit a combinatoric interpretation as weighted overpartition counts.
\end{abstract}

\maketitle


The partition crank function was developed by Andrews and Garvan \cite{Garvan-trip, AGCrank} with the goal of giving a combinatorial proof of Ramanujan's congruence \cite{Ramanujan}
\begin{equation} 
\label{Ram}
p(11n+6) \equiv 0 \mod 11,
\end{equation}
where $p(n)$ denotes the number of partitions of $n$. Recall, a partition $\lambda$ of $n$ is a non-increasing sequence of positive integers which sum to $n$. We denote this relation by $\lambda \vdash n$.

Given $\lambda$, let $\omega(\lambda)$ denote the number of occurrences of $1$ as a part of $\lambda$.
The \emph{crank} of $\lambda$, which we denote by $cr(\lambda)$, is then defined according to the value of  $\omega(\lambda)$.
If $\omega(\lambda) = 0$, then $cr(\lambda)$ is the largest part of $\lambda$.
Otherwise\footnote{There is a traditional handwaving with $\lambda = (1)$ in order to accommodate the generating series in \eqref{Garvan-prod}. This is inherited by the residual cranks which follow.}, $cr(\lambda)$ is equal to the number of parts of $\lambda$ which exceed $\omega(\lambda)$, minus $\omega(\lambda)$.
Let $M(m,n)$ denote the number of partitions $\lambda \vdash n$ with $cr(\lambda) = m$.
Andrews and Garvan proved that
\[
 M(m, 11n + 6) = \frac{1}{11} \ p(11n + 6),
\]
which is sufficient to demonstrate \eqref{Ram}.
The two-variable generating series for the cranks of partitions is given by \cite{AGCrank}
\begin{equation}             
\label{Garvan-prod}
C(z;q)\de \sum_{n\geq 0}\sum_{m\geq 1} M(m, n) z^m q^n =\frac{\aqprod{q}{q}{\infty}}{\aqprod{zq, q/z}{q}{\infty}},
\end{equation}
where 
\begin{gather*}
 \aqprod{\nklist{a}{k}}{q}{n} = \prod_{i=0}^{n-1} (1-a_1 q^i) \cdots (1-a_k q^i), \\
 \aqprod{\nklist{a}{k}}{q}{\infty}=\lim_{n\to\infty}\aqprod{\nklist{a}{k}}{q}{n}
\end{gather*}
is the $q$-Pochhammer symbol. From \eqref{Garvan-prod} we observe that
\begin{equation}
\label{mirror}
M(m, n) = M(-m, n).                    
\end{equation}
The crank function has gone on to play an interesting role in the theory of partitions, some examples being the moments of the crank function, and weighted counts of partitions \cite{ACK, run-run-Lovejoy}.
Let $C_\ell(q)$ denote the generating series for the $\ell$th moments of the crank function,
\[
  C_\ell(q) = \sum_{n=0}^\infty \sum_{m = -\infty}^\infty m^\ell M(m,n)q^n.
\]
Futher, take $\delta_q$ to be the usual differential operator $q \tfrac{\dif{}}{\dif{q}}$, and $P(q)\de \tfrac{1}{\aqprod{q}{q}{\infty}}$ the generating function for partitions. Atkin and Garvan established quasimodular properties of the function $C_\ell(q)$ and its images under $\delta_q$.

\begin{theorem}[Theorem 4.2 of \cite{AtkinGarvan}]
For $m\geq 0$ and $j\geq1$, the function $\delta_q^m \left(C_{2j}\right)$ is an element of the space
\[
P\cdot \overline{\mathcal{W}}_{m+j}\left(\Gamma\right),
\]
where $\overline{\mathcal{W}}_k\left(\Gamma\right)$ is the space of all quasimodular forms of weight at most $2k$ on $\Gamma$ with no constant term.
\end{theorem}

Loosely speaking, a quasimodular form is the product of a modular form which transforms under a group of M\"{o}bius transformations and the non-modular Eisenstein series $E_2(z)$. We postpone precise definitions until Section \ref{review}.

Of particular note is the notion of a \emph{residual crank}, wherein the ordinary crank function is applied to a subset of the parts of $\lambda$. Jennings-Shaffer \cite{Jennings} studied the functions
\begin{gather*}
  C1(z;q) = \frac{\aqprod{q^2}{q^4}{\infty} \aqprod{q}{q}{\infty}}
                  {\aqprod{zq, q/z}{q}{\infty}}, \\
  C2(z;q) = \frac{\aqprod{-q}{q^2}{\infty} \aqprod{q^2}{q^2}{\infty}}
                  {\aqprod{zq^2, q^2/z}{q^2}{\infty}}, \\
  C4(z;q) = \frac{\aqprod{q^4}{q^4}{\infty} }
                  {\aqprod{q}{q^2}{\infty} \aqprod{zq^4, q^4/z}{q^4}{\infty}},
\end{gather*}
as the generating series for residual crank-like functions for partitions without repeated odd parts.
If we let $C1_\ell(q)$, $C2_\ell(q)$, and $C4_\ell(q)$ denote the generating functions of the $\ell$th moments of these residual cranks, respectively, then these generating functions exhibit quasimodular properties.

\begin{theorem}[Theorem 1.1 of \cite{Jennings}]
For $k\geq 1$, the functions in 
\begin{gather*}
    \left\{\delta_q^m \left(C1_{2j}\right)\colon m \geq 0, 1\leq j\leq k, j + m \leq k \right\}, \\
    \left\{\delta_q^m \left(C2_{2j}\right)\colon m \geq 0, 1\leq j\leq k, j + m \leq k \right\}, \\
    \left\{\delta_q^m \left(C4_{2j}\right)\colon m \geq 0, 1\leq j\leq k, j + m \leq k \right\},
\end{gather*}
are in the space $\aqprod{-q}{q^2}{\infty}P\left(q^2\right)\overline{\mathcal{W}}_k\left(\Gamma_0(4)\right)$.
\end{theorem}

The crank function has been extended to the more general setting of overpartitions. An overpartition $\lambda$ is a non-increasing sequence of positive integers in which the first occurrence of any integer part may be overlined. We retain the notation $\lambda \vdash n$ to indicate that the sum of the integer parts of $\lambda$ is $n$. As an example, the overpartitions $\lambda \vdash 3$ are
\[
  (3),\  (\overline{3}),\  (2, 1),\  (2, \overline{1}),\ (\overline{2}, 1),\  (\overline{2}, \overline{1}),\  (1, 1, 1),\  (\overline{1}, 1, 1),
\]
which includes the three ordinary partitions of $3$. These objects were introduced by Corteel and Lovejoy \cite{Opartns} in order to provide a combinatorial interpretation of $q$-hypergeometric series.

Two residual crank functions have been defined in the overpartition setting by Bringmann, Lovejoy, and Osburn \cite{Bringmann}.
To calculate the \emph{first residual crank} of the overpartition $\lambda$, take $\lambda'$ to be the partition whose parts are the non-overlined parts of $\lambda$.
We then define $\ocr_1(\lambda) := cr(\lambda')$.

Similarly, to calculate the \emph{second residual crank} of the overpartition $\lambda$, take $\lambda'$ to be the partition whose parts are the non-overlined even parts of $\lambda$, divided by two. We then define $\ocr_2(\lambda) := cr(\lambda')$.
As an example, $\ocr_1(4, \overline{2}, 1) = cr(2, 1) = 0$ and $\ocr_2(4, \overline{2}, 1) = cr(2) = 2$. From \eqref{Garvan-prod}, we see that the two variable generating series for the first and second residual cranks of overpartitions are given by
\begin{equation}
\label{eq:frc}
\overline{C[1]}(z;q)
= \sum_{n=0}^\infty \sum_{m = -\infty}^\infty \overline{M[1]}(m,n) z^m q^n
= \frac{\aqprod{-q, q}{q}{\infty}}{\aqprod{zq, q/z}{q}{\infty}}
\end{equation}
and
\begin{equation}
\label{eq:src}
\overline{C[2]}(z;q)
=
\sum_{n=0}^\infty \sum_{m = -\infty}^\infty \overline{M[2]}(m,n) z^m q^n,
= \frac{\aqprod{-q}{q}{\infty}}{\aqprod{q}{q^2}{\infty}}
\frac{\aqprod{q^2}{q^2}{\infty}}{\aqprod{zq^2, q^2/z}{q^2}{\infty}} 
\end{equation}
respectively.

Let $\overline{C[1]}_\ell$ and $\overline{C[2]}_\ell$ denote the generating series for the $\ell$th moment of the first and second residual crank functions, respectively. Also, let $\overline{P}(q)\de \aqprod{-q}{q}{\infty}P(q)$ denote the generating function for overpartitions. Bringmann, Lovejoy, and Osburn established quasimodular properties for these moment generating functions.

\begin{theorem}[Theorem 1.1 of \cite{Bringmann}]
\label{thm:BLO}
For $k\geq 1$, the functions in 
\begin{gather*}
\left\{\delta_{q}^{m}\left(\overline{C[1]}_{2j}\right)\colon m\geq0,1\leq j\leq k,j+m\leq k\right\}, \\
\left\{\delta_{q}^{m}\left(\overline{C[2]}_{2j}\right)\colon m\geq0,1\leq j\leq k,j+m\leq k\right\},
\end{gather*}
are in the space $\overline{P}\cdot\overline{\mathcal{W}}_{k}(\Gamma_0(2))$.
\end{theorem}

Work of Al-Saedi, Swisher, and the first author \cite{AMS} extends the first and second residual cranks by defining the $k$th residual crank functions for all $k>1$.
To calculate the $k$th residual crank of the overpartition $\lambda$, take $\lambda'$ to be the partition whose parts are the non-overlined parts of $\lambda$ which vanish modulo $k$, divided by two.
We then define $\ocr_k(\lambda) := cr(\lambda')$.

This definition generalizes the two-variable generating functions \eqref{eq:frc} and \eqref{eq:src} with the family
\[
\overline{C[k]}\left(z;q\right) \de \left(q^k;q^k\right)_{\infty}\overline{P}(q)C\left(z;q^k\right) = \sum_{n= 0}^{\infty}\sum_{m=-\infty}^{\infty}\overline{M[k]}(m,n)z^mq^n,
\]
whose moment generating functions are calculated via
\[
\overline{C[k]}_{\ell}(q) \de \left. \delta_{z}^{\ell}\left(\overline{C[k]}\left(z;q\right)\right)\right|_{z=1} = \sum_{n=0}^{\infty}\sum_{m=-\infty}^{\infty} m^{\ell}\overline{M[k]}(m,n)q^n.
\]
Following the works of \cite{AtkinGarvan,Jennings,Bringmann}, we will prove the following generalization of Theorem \ref{thm:BLO}.

\begin{theorem}
\label{thm:mainQuasimod}
For $l\geq 1$ and $k\geq 1$, the functions in
\[
\left\{\delta_{q}^{m}\left(\overline{C[k]}_{2j}\right)\colon m\geq0,1\leq j\leq l,j+m\leq l\right\}
\] 
are in the space $\overline{P}\cdot\overline{\mathcal{W}}_{l}\left(\Gamma_0\left(\lcm{2,k}\right)\right)$, where $\lcm{2,k}$ is the least common multiple of $2$ and $k$.
\end{theorem}

Of course, quasimodularity is not the only focus in the study of crank functions.
Let $\overline{M[1]}_\ell(n)$ and $\overline{M[2]}_\ell(n)$ denote the $\ell$th crank moments of the first and second residual crank functions,
\begin{align*}
\overline{M[1]}_\ell(n) = \sum_{m = -\infty}^\infty m^\ell \  \overline{M[1]}(m,n), \\
\overline{M[2]}_\ell(n) = \sum_{m = -\infty}^\infty m^\ell \  \overline{M[2]}(m,n).
\end{align*}
Larson, Rust, and Swisher proved in the following inequality for these moments in the case $\ell=1$.
\begin{theorem}[\cite{Swish}]
For all $n \geq 2$ we have $\overline{M[1]}_{1}(n) > \overline{M[2]}_{1}(n)$.
\end{theorem}

This result generalizes to all the $k$th residual cranks of overpartitions.
\begin{theorem}[\cite{AMS}]
\label{thm:momentsIneq}
  For all $k, \ell \geq 1$,
  \[
    \overline{M[k+1]}_\ell^+(n) \leq \overline{M[k]}_\ell^+(n),
  \]
  with inequality if and only if $n < 2k$, in which case $\overline{M[k]}_\ell^+(n) = 0$.
\end{theorem}

Our combinatorial interpretation of $\overline{M[k]}_\ell^+(n)$ allows us to improve Theorem \ref{thm:momentsIneq} in the case $\ell=2$.

\begin{cor}
    For all $k, m \geq 1$,
  \[
    m \cdot \overline{M[m d]}_2^+(n) \leq \overline{M[k]}_2^+(n),
  \]
  with inequality if and only if $n < 2k$, in which case $\overline{M[d]}_\ell^+(n) = 0$.
\end{cor}

The rest of this paper is organised as follows. In Section \ref{review}, we review properties of quasimodular forms on the congruence group $\Gamma_0(N)$. Section \ref{quasimodular} is devoted to the proof of Theorem \ref{thm:mainQuasimod}. In Section \ref{count}, we give a combinatorial interpretation of the moments $\overline{M[k]}_2(n)$ and explore some corollaries. Finally, we close in Section \ref{end} with remarks for future study.

\section{Quick Review of Modular and Quasimodular Forms}
\label{review}

We begin with a review of quasimodular forms, themselves an extension of the more ordinary modular forms. More complete details on the theory of modular forms may be found in Apostol \cite{Apostol} for an introduction,  and Miyake or Ono \cite{Miyake,Ono} for more advanced material.

For $N\in\mathbb{N}$, we define $\Gamma_0(N)$ to be the group of transformations given by
\begin{equation}
\label{eq:moebius}
\Gamma_0(N)\de \left\{\tau\mapsto \frac{a\tau+b}{c\tau+d}\colon a,b,c,d\in\mathbb{Z},ad-bc=1,c\equiv 0\; (\mathrm{mod}\;N)\right\}.
\end{equation}
Here, each element $\gamma \in \Gamma_0(N)$ is a M\"obius transformation acting on the complex upper half plane $\mathbb{H} = \{\tau : \Im \tau  > 0 \}$, and may be represented as an integral matrix\footnote{This representation is taken modulo $-I$.}
\[
 \gamma = \begin{pmatrix} a& b \\ c & d \end{pmatrix}.
\]

In this notation, $\Gamma\de\Gamma_0(1)$ is the modular group.
For $N>1$, we call $\Gamma_0(N)$ the \emph{principal congruence subgroup} of \emph{level} $N$.
It will be useful to recall that $\Gamma$ is generated by the transformations
\[
  \tau \mapsto \tau + 1 \qquad \tau \mapsto \frac{-1}{\tau}.
\]

Next, let $k$ be an integer.
A holomorphic function $f: \mathbb{H} \to \CC$ is a \emph{modular form of weight $k$ on $\Gamma_0(N)$} if the following conditions are met.
Fist, 
\[
f\left(\gamma(\tau)\right)=\left(c\tau+d\right)^k f(\tau)
\]
for all $\tau\in\mathbb{H}$ and $\gamma\in\Gamma_0(N)$.
Second, $f$ has a representation as a power series in the variable $q\de 
e^{2\pi\ie\tau}$, e.g.
\[
f(\tau)=\sum_{n=0}^\infty a_n q^n.
\]
Note that this imposes growth conditions on $f$ as $\tau \to i \infty$.

It is a fact that if $f$ is a modular form of weight $k$, then $k$ is even.
If $k=0$, then $f$ is constant by Liouville's theorem, and if $k=2$, then $f\equiv0$. Nonconstant modular forms $f$ must have weight $k \geq 4$, see Apostol \cite[Theorem 6.2]{Apostol}.
We define $\mathcal{M}_k\left(\Gamma_0(N)\right)$ to be the vector space of modular forms of weight $2k$ on $\Gamma_0(N)$, where $k$ is a non-negative integer. 

Following Atkin and Garvin \cite{AtkinGarvan}, the \emph{Eisenstein series} $E_{2k}(\tau)$ is defined for $k\in\mathbb{N}$ by 
\[
E_{2k}(\tau) = 1+\frac{(2\pi)^{2k}}{(-1)^k(2k-1)!\zeta(2k)}\Phi_{2k-1}(q),
\]
where
\[
\Phi_{l}(q)\de \sum_{n=1}^{\infty}\frac{n^lq^n}{1-q^n}=\sum_{n=1}^{\infty}\sum_{d\mid n}d^lq^n.
\]
For all $k\geq2$, we have $E_{2k}\in\mathcal{M}_k(\Gamma)$.
In fact, these functions generate all modular forms, as
\[
\left\{E_4^{a}E_6^{b}\colon a,b\in\mathbb{N}_0, 2a+3b=k\right\}
\]
forms a basis for $\mathcal{M}_k(\Gamma)$, as demonstrated by Apostol \cite[Section 6.5]{Apostol}. 

Note that $E_2=1-24\Phi_1$ is not a modular form, since 
\begin{equation}
\label{eq:E2}
E_2\left(-\frac{1}{\tau}\right)=\tau^2E_2(\tau)+\frac{6}{\pi\mathrm{i}}\tau
\neq (\tau)^2 E_2(\tau).
\end{equation}

Quasimodular forms were first defined in work of Kaneko and Zagier \cite{KanekoZagier}.
Here we extend the definition of weight by declaring the weight of $E_2$ to be $2$.
We can then define the space of \emph{quasimodular forms of weight $2k$ on $\Gamma_0(N)$} by 
\[
\overline{\mathcal{M}}_k\left(\Gamma_0(N)\right) \de \left\{\sum_{j=0}^{k}f_jE_2^j\colon f_j\in\mathcal{M}_{k-j}\left(\Gamma_0(N)\right), f_k\not\equiv0\right\}.
\]
We see immediately that $\mathcal{M}_k\left(\Gamma_0(N)\right)$ is a subspace of $\overline{\mathcal{M}}_k\left(\Gamma_0(N)\right)$. Because $\delta_q\left(E_2\right)=\frac{1}{12}\left(E_2^2-E_4\right)$, we have $\delta_q\left(E_2\right)\in\overline{\mathcal{M}}_2\left(\Gamma\right)$. 

We now define $\mathcal{W}_{k}\left(\Gamma_0(N)\right)$ to be the space of \emph{quasimodular forms of weight at most $2k$ on $\Gamma_0(N)$}, that is
\[
\mathcal{W}_{k}\left(\Gamma_0(N)\right) \de \left\{\sum_{j=0}^k f_jE_2^j\colon f_j\in\sum_{l=0}^{k-j}\mathcal{M}_l\left(\Gamma_0(N)\right)\right\}.
\]
Then $\overline{\mathcal{M}}_k\left(\Gamma_0(N)\right)$ is a vector subspace of $\mathcal{W}_{k}\left(\Gamma_0(N)\right)$, and we see that the set
\[
\left\{E_2^{a}E_4^{b}E_6^{c}\colon a,b,c\in\mathbb{N}_0, a+2b+3c\leq k\right\}
\]
forms a basis for $\mathcal{W}_k(\Gamma)$.
Finally, define the space $\overline{\mathcal{W}}_k\left(\Gamma_0(N)\right)$ via
\[
  \overline{\mathcal{W}}_k\left(\Gamma_0(N)\right)
  =
  \{
    f=\sum_{n=0}^{\infty} a_n q^n \in \overline{\mathcal{W}}_k\left(\Gamma_0(N)\right) | \ a_0 = 0
  \}.
\]
We can see that $\Phi_{2k-1}(q)\in\overline{\mathcal{W}}_k\left(\Gamma\right)$.

\section{Proof of Theorem \ref{thm:mainQuasimod}}
\label{quasimodular}

We next give some useful lemmas towards the proof of Theorem \ref{thm:mainQuasimod}.
These lemmas are well-known to specialists, but we provide their proofs for the sake of completeness.

\begin{lemma}
\label{lem:ntau}
Let $n\in\mathbb{N}$ and $f(\tau)\in\mathcal{M}_{k}\left(\Gamma_0(N)\right)$. Then $f(n\tau)\in\mathcal{M}_{k}\left(\Gamma_0(nN)\right)$.
\end{lemma}

\begin{proof}
It is sufficient to demonstrate that $f(n\gamma(\tau))=(c\tau+d)^{2k} f(n\tau)$, where
\[
\gamma(\tau)\de \left( \frac{a\tau+b}{c\tau+d}\right),
\]
$ad-bc=1$ and $c\equiv 0\; (\mathrm{mod}\;nN)$ for integers $a,b,c,d$. There exists an integer $c'$ such that $c=c'nN$. Then we have $n\gamma(\tau)=\gamma_1(n\tau)$, where $\gamma_1(\tau)\de \left(a\tau+nb\right)/\left(c'N\tau+d\right)$. Because $\gamma_1\in\Gamma_0(N)$, we obtain
\[
f\left(\gamma_1(n\tau)\right)=\left(c'nN\tau+d\right)^{2k}f(n\tau),
\]
which is our claim. 
\end{proof}

A special case of Lemma \ref{lem:ntau} is that $f(\tau)\in\mathcal{M}_k(\Gamma)$ implies  $f(n\tau)\in\mathcal{M}_k\left(\Gamma_0(n)\right)$. We also deduce that $\Phi_{2k-1}\left(q^n\right)\in\overline{\mathcal{W}}_k\left(\Gamma_0(n)\right)$.

\begin{lemma}
\label{lem:lifting}
Let $f\in\mathcal{M}_k\left(\Gamma_0(N)\right)$. Then $12\delta_q(f)-2kE_2f\in\mathcal{M}_{k+1}\left(\Gamma_0(N)\right)$.
\end{lemma}

\begin{proof}
For brevity, we let
\[
g(\tau) \de 12\delta_q(f(\tau))-2kE_2(\tau)f(\tau) = \frac{6}{\pi\mathrm{i}}f'(\tau)-2kE_2(\tau)f(\tau).
\]
It suffices to prove that $g$ is modular of weight $k+1$ for both generators of $\Gamma$.
We already have $g(\tau+1)=g(\tau)$.

Let $\tau=-1/\tau'$.
Using the fact that $f\left(-1/\tau'\right)=\tau'^{2k}f\left(\tau'\right)$ and \eqref{eq:E2} gives us
\begin{gather*}
f'(\tau)=\tau'^2\frac{\dif{}}{\dif{\tau'}}\left(\tau'^{2k}f\left(\tau'\right)\right)=\left(2kf\left(\tau'\right)+\tau'f'\left(\tau'\right)\right)\tau'^{2k+1}, \\
E_2(\tau)f(\tau) = \left(\tau'E_2\left(\tau'\right)f\left(\tau'\right)+\frac{6}{\pi\mathrm{i}}\right)\tau'^{2k+1}.
\end{gather*}
Therefore, $g(-1/\tau)=\tau^{2k+2}g(\tau)$.
\end{proof}

\begin{lemma}
\label{lem:lifting2}
Let $k\geq1$. If $f\in\overline{\mathcal{M}}_k\left(\Gamma_0(N)\right)$, then $\delta_q(f)\in\overline{\mathcal{M}}_{k+1}\left(\Gamma_0(N)\right)$. 
\end{lemma}

\begin{proof}
Let $f\in\overline{\mathcal{M}}_k\left(\Gamma_0(N)\right)$. Then there exist modular forms $f_j\in\mathcal{M}_{k-j}\left(\Gamma_0(N)\right)$ such that $f=\sum_{j=0}^k f_jE_2^j$ and $f_k\not\equiv0$. By Lemma \ref{lem:lifting} we have $\delta_q\left(f_j\right)=f_{0,j}+(k/6)f_{j}E_2$ for some $f_{0,j}\in\mathcal{M}_{k+1-j}\left(\Gamma_0(N)\right)$.

We need to show that $\delta_q(f)=\sum_{j=0}^{k+1}\tilde{f}_jE_2^j$ for some modular forms $\tilde{f}_j\in\mathcal{M}_{k+1-j}\left(\Gamma_0(N)\right)$.
We break into three cases.

Suppose $k=1$.
We have $\tilde{f}_0=f_{0,0}-\frac{1}{12}f_1E_4$, $\tilde{f}_1=f_{0,1}+\frac{1}{6}f_0$ and $\tilde{f}_2=\frac{1}{4}f_1$.

Suppose $k=2$.
We have $\tilde{f}_0=f_{0,0}-\frac{1}{12}f_1E_4$, $\tilde{f}_1=f_{0,1}+\frac{1}{3}f_0-\frac{1}{6}f_2E_4$, $\tilde{f}_2=f_{0,2}+\frac{5}{12}f_1$ and $\tilde{f}_3=\frac{1}{2}f_2$.

Otherwise, $k\geq3$.
We have $\tilde{f}_0=f_{0,0}-\frac{1}{12}f_1E_4$, $\tilde{f}_1=f_{0,1}+\frac{k}{6}f_0-\frac{1}{6}f_2E_4$, 
\[
\tilde{f}_{j}=\sum_{l=2}^{k-1} f_{0,l}+\frac{2k+l-1}{12}f_{l-1}-\frac{l+1}{12}f_{l+1}E_4
\]
for $2\leq j\leq k-1$, $\tilde{f}_k=f_{0,k}+\frac{k-1}{4}f_{k-1}$ and $\tilde{f}_{k+1}=\frac{k}{4}f_k$.

\end{proof}

\begin{cor}
  If $f\in\mathcal{W}_{k}\left(\Gamma_0(N)\right)$, then $\delta_q(f)\in\overline{\mathcal{W}}_{k+1}\left(\Gamma_0(N)\right)$.
\end{cor}

Note that the moment generating function $\overline{C[k]}_{\ell}(q)$ vanished for odd $\ell$, since $\overline{M}_k(m,n) = \overline{M}_k(-m,n)$

Following \cite{AtkinGarvan}, we can obtain the representation
\begin{equation}
\label{eq:rep}
\overline{C[k]}_{2j}(q) = 2\overline{P}(q)\sum_{a_1+2a_2+\cdots+ja_j=j}\alpha_{a_1,\ldots,a_j}\Phi_1^{a_1}\left(q^k\right)\Phi_3^{a_2}\left(q^k\right)\cdots\Phi_{2j-1}^{a_j}\left(q^k\right),
\end{equation}
where $\alpha_{a_1,\ldots,a_j}$ are integers.

Let $K\de \lcm{2,k}$. By \eqref{eq:rep} and Lemma \ref{lem:ntau} we have $\overline{C[k]}_{2j}(q)\in \overline{P}\cdot\overline{\mathcal{W}}_{j}\left(\Gamma_0(K)\right)$, and by Lemma \ref{lem:lifting2} also $\delta_q\colon \overline{\mathcal{W}}_{l}\left(\Gamma_0(K)\right)\to\overline{\mathcal{W}}_{l+1}\left(\Gamma_0(K)\right)$. Because
\[
\delta_q\left(\overline{P}\right) = 2\overline{P}\left(\Phi_1(q)-\Phi_1\left(q^2\right)\right)\in \overline{P}\cdot\overline{\mathcal{W}}_{1}\left(\Gamma_0(K)\right), 
\]
we also have $\delta_q\colon \overline{P}\cdot\overline{\mathcal{W}}_{j}\left(\Gamma_0(K)\right)\to\overline{P}\cdot\overline{\mathcal{W}}_{j+1}\left(\Gamma_0(K)\right)$. Therefore,
\[
\delta_{q}^{m}\left(\overline{C[k]}_{2j}(q)\right) \in \overline{P}\cdot\overline{\mathcal{W}}_{j+m}\left(\Gamma_0(K)\right) \subset \overline{P}\cdot\overline{\mathcal{W}}_{l}\left(\Gamma_0(K)\right).
\]
This proves Theorem \ref{thm:mainQuasimod}.

\section{Weighted Overpartition Counts}
\label{count}

We now establish a combinatoric interpretation for the coefficients of $\overline{C[k]}_2(q)$.
Let $nov_k(n)$ denote the sum of all non-overlined parts which vanish modulo $k$, taken across all overpartitions $\lambda \vdash n$.
For example, the overpartitions $\lambda \vdash 3$ are given by
\begin{align*}
  (3),\  (\overline{3}),\  ({2}, 1),\  ({2}, {1}),\ (\overline{2}, 1),\  ({2}, \overline{1}),\  (1, 1, 1),\  (\overline{1}, 1, 1).
\end{align*}
We see that $nov_2(3) = 6$.
Similarly, let $ov_k(n)$ denote the sum of all overlined parts which vanish modulo $k$, taken across all $\lambda \vdash n$. 
Here, $ov_2(3) = 2$.
Note that $nov_2(n) = enov(n)$ in the notation of \cite[p.~1768]{Bringmann}.

\begin{theorem}
We have
\[
nov_k(n) = \frac{k}{2} \cdot \overline{M[k]}_2.
\]
\end{theorem}

\begin{proof}
We use Dyson's result \cite{Dyson89} that
\[
  n p(n) = \frac{1}{2} M_2(n).
\]
In the case of residual cranks, we have
  \begin{align*}
    \sum_{n = 0}^\infty nov_k(n) q^n &= \frac{\aqprod{-q}{q}{\infty}\aqprod{q^k}{q^k}{\infty}}{\aqprod{q}{q}{\infty}} \sum_{n = 0}^\infty k n \cdot p(n) q^{kn}\\
    &= \frac{\aqprod{-q}{q}{\infty}\aqprod{q^k}{q^k}{\infty}}{\aqprod{q}{q}{\infty}} \frac{k}{2} \sum_{n = 0}^\infty M_2(n) q^kn\\
    &= \frac{\aqprod{-q}{q}{\infty}\aqprod{q^k}{q^k}{\infty}}{\aqprod{q}{q}{\infty}} \frac{k}{2} \ \delta_z^2 \bigg\{ \frac{\aqprod{q^k}{q^k}{\infty}}{\aqprod{zq^k, q/z^k}{q^k}{\infty}} \bigg\} _{z=1} \\
    &= \frac{k}{2} \cdot \overline{C[k]}_2(q).
  \end{align*}
\end{proof}

\begin{lemma}
We have
\[
ov_k (n) = nov_k(n) - nov_{2k}(n).
\]
\end{lemma}

\begin{proof}
We claim there is a weight-preserving bijection $\phi_k$ between partitions into distinct parts which vanish modulo $k$, and partitions whose parts vanish modulo $k$ but not modulo $2k$.
This map is simply a dilation of Euler's map between partitions into distinct parts and partitions into odd parts \cite{Abook}.
Thus, summing the overlined parts which vanish modulo $k$, when taken over all $\lambda \vdash n$, is equivalent to summing the non-overlined parts which vanish modulo $k$ but not modulo $2k$, as desired.
\end{proof}

\begin{cor}
We have
\[
ov_k (n) = \frac{k}{2} \cdot \overline{M[k]}_2(n) - k \cdot \overline{M[2k]}_2(n).
\]
\end{cor}


\begin{cor}
  For all $d, k \geq 1$ and $n \geq 0$,
  \[
    d \overline{M[dk]}(n) \leq \overline{M[k]}(n),
  \]
  with equality if and only if $n < k$.
\end{cor}

\begin{proof}
Since $nov_{dk}(n) \leq nov_{k}(n)$, we have
\[
0 \leq nov_d(n) - nov_{dk}(n) = \frac{k}{2} \left(\overline{M[k]}_2(n) -  d\overline{M[dk]}_2(n)\right).
\]
The conditions for equality follow directly from Theorem 1 of \cite{AMS}.
\end{proof}

Finally, we show how the second moment of the $k$th residual crank relates to weighted counts of all overpartitions.
For an overpartition $\lambda$, let $\omega_k(\lambda)$ denote the number of times $k$ occurs non-overlined as a part in $\lambda$.

\begin{cor}
  \[
  \overline{C[k]}_2(q)
  =
  - \sum_{\lambda \in \OP} \omega_k (\lambda) \ocr_k(\lambda) q^{|\lambda|}. 
  \]
\end{cor}

\begin{proof}
  We use Chern's result \cite{Chern} that
  \[
    \sum_{\lambda \vdash n} \omega(\lambda) cr(\lambda) = -n p(n).
  \]
  Then
\begin{align*}
    \overline{C[k]}_2(q)
    &= \frac{\aqprod{-q}{q}{\infty}\aqprod{q^k}{q^k}{\infty}}{\aqprod{q}{q}{\infty}} \sum_{n = 0}^\infty M_2(n) q^kn \\
    &= \frac{\aqprod{-q}{q}{\infty}\aqprod{q^k}{q^k}{\infty}}{\aqprod{q}{q}{\infty}} \sum_{n = 0}^\infty n \cdot p(n) q^{kn} \\
    &= \frac{\aqprod{-q}{q}{\infty}\aqprod{q^k}{q^k}{\infty}}{\aqprod{q}{q}{\infty}} \sum_{n = 0}^\infty \sum_{\substack{\lambda' \vdash n \\ \lambda' \in P}} -\omega(\lambda') cr(\lambda') q^{kn},
\end{align*}
where the inner summation is taken over ordinary partitions $\lambda'$.
If we consider such $\lambda'$ to be the residual partitions as used to define the $k$th residual crank, then we see that
\[
  \frac{\aqprod{-q}{q}{\infty}\aqprod{q^k}{q^k}{\infty}}{\aqprod{q}{q}{\infty}} \sum_{n = 0}^\infty \sum_{\substack{\lambda' \vdash n \\ \lambda' \in P}} -\omega(\lambda') cr(\lambda') q^{kn}
  = - \sum_{\lambda \in \OP} \omega_k (\lambda) \ocr_k(\lambda) q^{|\lambda|},
\]
as desired.
\end{proof}

\section{Future Study}
\label{end}

Those familiar with the theory are likely to ask how the quasimodularity of the $k$th residual crank functions interacts with moment generating functions of any corresponding overpartition rank functions.
It is common to study crank functions in relation to certain rank functions defined for partitions and overpartitions, which will be nearly quasimodular on the same group $\Gamma_0(N)$ \cite{AtkinGarvan, Jennings, Bringmann}. For example, Bringmann, Lovejoy and Osburn established a partial differential equation between the first and second residual crank functions, and the Dyson rank and $M_2$-rank functions \cite{Automorphic}.
This work then led to establishing the quasimodularity of  functions such as
\begin{multline*}
  (\ell^2 - 3\ell + 2) \overline{R}_\ell + 2 \sum_{i=1}^{\ell/2-1} \binom{\ell}{2i}(3^{2i} - 2^{2i} - 1) \delta_q \overline{R}_{\ell - 2i} \\
  + \sum_{i=1}^{\ell/2-1} \bigg(      
    \binom{\ell}{2i}(2^{2i} + 1)
    + 2 \binom{\ell}{2i+1} (1 - 2^{2i+1})
    + \frac{1}{2}\binom{\ell}{2i+2} (3^{2i+2} - 2^{2i+2} - 1)
  \bigg) \overline{R}_{\ell - 2i},
\end{multline*}
where $\overline{R}_\ell = \overline{R}_\ell(q)$ is the generating series for the $\ell$th moment of the Dyson rank function for overparititons \cite{Bringmann}.

At this point in time, we are not aware of a suitable rank function to pair with the $k$th residual crank function for $k >2$. Although the first author's $M_k$-ranks \cite{Morrill} generalize the rank functions studied by Bringmann, Lovejoy, and Osburn, these fail to produce the expected $spt$ relations.
Where we would expect
\[
spt_k(n) = \frac{1}{2}\left(\overline{N[k]}_2(n) - \overline{M[k]}_2(n)\right),
\]
with $spt_k(n)$ a positive-weighted count function of overpartitions, we have instead
\[
\frac{1}{2}\left(\overline{N[3]}_2(4) - \overline{M[3]}_2(4)\right) = {-2}.
\]

It may be fruitful to instead establish a theory of $k$th residual $spt$ functions for overpartitions, then seek rank functions so that
\[
  \overline{N[k]}_2(n) = \overline{M[k]}_2(n) - 2 spt_k(n).
\]

\section{Acknowledgements}

The first author is supported by Australian Research Council Discovery Project DP160100932.


\providecommand{\bysame}{\leavevmode\hbox to3em{\hrulefill}\thinspace}
\providecommand{\MR}{\relax\ifhmode\unskip\space\fi MR }
\providecommand{\MRhref}[2]{%
  \href{http://www.ams.org/mathscinet-getitem?mr=#1}{#2}
}
\providecommand{\href}[2]{#2}


\begin{thebibliography}{ASMS0}

\bibitem[ACK13]{ACK}
G.~E. Andrews, S.~H. Chan, and B.~Kim, \emph{The odd moments of ranks and
  cranks}, J. Combin. Theory Ser. A \textbf{120} (2013), no.~1, 77--91.

\bibitem[AG88]{AGCrank}
G.~E. Andrews and F.~G. Garvan, \emph{Dyson's crank of a partition}, Bull.
  Amer. Math. Soc. (N.S.) \textbf{18} (1988), no.~2, 167--171.

\bibitem[AG03]{AtkinGarvan}
A.~O.~L. Atkin and F.~G. Garvan, \emph{Relations between the ranks and cranks
  of partitions}, Ramanujan J. \textbf{7} (2003), no.~1-3, 343--366.

\bibitem[And76]{Abook}
G.~E. Andrews, \emph{The Theory of Partitions}, Addison-Wesley Publishing Co.,
  Reading, Mass.-London-Amsterdam, 1976, Encyclopedia of Mathematics and its
  Applications, Vol. 2.

\bibitem[Apo90]{Apostol}
T.~M. Apostol, \emph{Modular Functions and {D}irichlet Series in Number
  Theory}, 2nd ed., Graduate Texts in Mathematics, vol.~41, Springer-Verlag,
  New York, 1990.

\bibitem[ASMS0]{AMS}
A.~Al-Saedi, T.~Morrill, and H.~Swisher, \emph{Inequalities for the dth
  residual crank moments of overpartitions}, International Journal of Number
  Theory, doi: 10.1142/S1793042120500840.

\bibitem[BLO09]{Bringmann}
K.~Bringmann, J.~Lovejoy, and R.~Osburn, \emph{Rank and crank moments for
  overpartitions}, J. Number Theory \textbf{129} (2009), no.~7, 1758--1772.

\bibitem[BLO10]{Automorphic}
\bysame, \emph{Automorphic properties of generating functions for generalized
  rank moments and {D}urfee symbols}, Int. Math. Res. Not. IMRN (2010), no.~2,
  238--260.

\bibitem[Che20]{Chern}
S.~Chern, \emph{Weighted partition rank and crank moments. i.{A}ndrews--{B}eck type congruences},
  \url{https://sites.psu.edu/shanechern/files/2019/03/Weighted-partition-rank-and-crank-moments-I-20wr6zy.pdf},
  2020, Preprint Accessed: 2020-05-04.

\bibitem[CKL09]{run-run-Lovejoy}
D.~Choi, S.-Y. Kang, and J.~Lovejoy, \emph{Partitions weighted by the parity of
  the crank}, J. Combin. Theory Ser. A \textbf{116} (2009), no.~5, 1034--1046.

\bibitem[CL04]{Opartns}
S.~Corteel and J.~Lovejoy, \emph{Overpartitions}, Trans. Amer. Math. Soc. (2004), no.~4, 1623--1635.

\bibitem[Dys89]{Dyson89}
F.~J.~Dyson, \emph{Mappings and symmetries of partitions}, J. Combin. Theory Ser. A \textbf{51} (1989), no.~2, 169--180.

\bibitem[Gar88]{Garvan-trip}
F.~G. Garvan, \emph{New combinatorial interpretations of {R}amanujan's
  partition congruences mod {$5,7$} and {$11$}}, Trans. Amer. Math. Soc.
  \textbf{305} (1988), no.~1, 47--77.

\bibitem[JS15]{Jennings}
C.~Jennings-Shaffer, \emph{Rank and crank moments for partitions without
  repeated odd parts}, Int. J. Number Theory \textbf{11} (2015), no.~3,
  683--703.

\bibitem[KZ95]{KanekoZagier}
M.~Kaneko and D.~Zagier, \emph{A generalized {J}acobi theta function and
  quasimodular forms}, The moduli space of curves ({T}exel {I}sland, 1994),
  Progr. Math., vol. 129, Birkh\"{a}user Boston, Boston, MA, 1995,
  pp.~165--172.

\bibitem[LRS14]{Swish}
A.~Larsen, A.~Rust, and H.~Swisher, \emph{Inequalities for positive rank and
  crank moments of overpartitions}, Int. J. Number Theory \textbf{10} (2014),
  no.~8, 2115--2133.

\bibitem[Miy06]{Miyake}
T.~Miyake, \emph{Modular Forms}, english ed., Springer Monographs in
  Mathematics, Springer-Verlag, Berlin, 2006, Translated from the 1976 Japanese
  original by Yoshitaka Maeda.

\bibitem[Mor19]{Morrill}
T.~Morrill, \emph{Two families of buffered {F}robenius representations of
  overpartitions}, Ann. Comb. \textbf{23} (2019), no.~1, 103--141.

\bibitem[Ono04]{Ono}
K.~Ono, \emph{The Web of Modularity: Arithmetic of the Coefficients of Modular Forms and {$q$}-series}, CBMS Reg.~Conf.~Ser.~Math., vol. 102, AMS, 2004.

\bibitem[Ram21]{Ramanujan}
S.~Ramanujan, \emph{Congruence properties of partitions}, Math. Z. \textbf{9}
  (1921), no.~1-2, 147--153.

\end{thebibliography}
\end{document}